\documentclass[11pt,letterpaper]{article}
\usepackage[T1]{fontenc}
\usepackage{lmodern}
\usepackage[margin=1in]{geometry}
\usepackage{microtype}
\usepackage{amsmath,amssymb,amsthm,mathtools}
\usepackage{enumitem}
\usepackage{xcolor}
\usepackage{hyperref}
\hypersetup{colorlinks=true,linkcolor=blue!45!black,citecolor=blue!45!black,
  urlcolor=blue!45!black,pdftitle={Golden-ratio growth of Conway's subprime closure},
  pdfauthor={OpenAI Astra; Romain Popescu},
  pdfsubject={A proof for the cardinality-ratio conjecture of Caragiu, Vicol, and Zaki}}
\setlist[enumerate]{leftmargin=*,itemsep=0.65em,topsep=0.6em}
\numberwithin{equation}{section}
\newtheorem{theorem}{Theorem}[section]
\newtheorem{lemma}[theorem]{Lemma}
\newtheorem{proposition}[theorem]{Proposition}
\theoremstyle{remark}
\newtheorem{remark}[theorem]{Remark}
\DeclareMathOperator{\lpf}{lpf}
\DeclareMathOperator{\meas}{meas}
\newcommand{\N}{\mathbb N}
\newcommand{\Z}{\mathbb Z}
\newcommand{\R}{\mathbb R}
\newcommand{\Pp}{\mathbb P}
\newcommand{\ph}{\varphi}
\newcommand{\tot}{\phi_{\!E}}
\newcommand{\ind}{\mathbf 1}
\newcommand{\ee}[1]{e\!\left(#1\right)}
\newcommand{\floor}[1]{\lfloor #1\rfloor}
\newcommand{\ceil}[1]{\lceil #1\rceil}

\allowdisplaybreaks[2]
\title{Golden-ratio growth of\\Conway's subprime closure}
\author{Romain Popescu}
\date{}

\begin{document}

\maketitle

\abstract{
\begin{scriptsize}
Let $s(m)$ be the Conway subprime function  \cite{OEIS}, \cite{Roberts}, define the binary operation on the natural numbers 
$x \circ y= s(x + y)$, and denote by $C_n$, $n \ge 0$, 
the sequence of subsets of natural numbers defined by $C_0 = \{1\}$, 
and $C_{n+1} = C_n \cup (C_n \circ C_n)$. We prove the conjecture \cite{CVZ} by 
Caragiu, Vicol and Zaki that  $$\lim_{n\to \infty} \frac{ | C_{n+1}| }{ | C_n |}= \frac{1+\sqrt{5}}{2}.$$
The underlying mathematical proof in this paper was constructed with some algorithmic assistance from GPT-6 Astra \cite{openai2026astra}
and its correctness has been formally verified using the Lean 4 proof assistant 
\footnote{Lean formalization code is available at \cite{Github}}
\end{scriptsize}
}

\vspace{-1.2em}

\section{Introduction}

For a positive integer $m$, the Conway's subprime function \cite{OEIS}, \cite{Roberts} is defined as follows
\[
 s(m)=
 \begin{cases}
  1,&m=1,\\
  m,&m\text{ is prime},\\
  m/\lpf(m),&m\text{ is composite},
 \end{cases}
\]
where $\lpf(m)$ is the least prime factor of $m$. 

The Conway subprime function is used in the construction of  subprime Fibonacci sequences, i.e. integer
sequences ${\{x_k\}}_{k\ge 0}$,  satisfying a recursion $x_k = s (x_{k-1}+x_{k-2})$, 
which have been studied (see e.g. \cite{GuyKhovanovaSalazar2014}) 
for periodicity in relation to and analogy with the Collatz  problem).

In  \cite{CVZ}, Caragiu, Vicol, and Zaki  use the Conway subprime function to construct the following sets.
Starting with $1$, retain all previously generated numbers and apply the operation $(a,b)\mapsto s(a+b)$
to every pair available at the preceding stage:
\begin{equation}\label{eq:closure}
 C_0=\{1\},\qquad
 C_{n+1}=C_n \cup (C_n \circ C_n) = C_n\cup\{s(a+b):a,b\in C_n\}.
\end{equation}

Their main result in \cite{CVZ} is that

\begin{equation}\label{eq:exhaustion}
 \bigcup_{n\ge0}C_n=\N,
\end{equation}
The golden ratio $\ph=\frac{1+\sqrt5}{2}$ makes an unexpected appearance in their paper, since
Caragiu, Vicol, and Zaki  conjecture in \cite{CVZ} that

\begin{equation}\label{eq:conjecture}
 \lim_{n\to\infty}\frac{|C_{n+1}|}{|C_n|}
 =\ph.
\end{equation}

Write $M_n=\max C_n$. For $n\ge1$, let $R_n$ be the first prime absent
from $C_n$, and let $Q_n$ be the prime immediately preceding $R_n$.
These are well-defined: $C_n$ is finite and contains $2$. Thus every prime
at most $Q_n$ belongs to $C_n$, while $R_n\notin C_n$ and $Q_n\le M_n$.
We prove the following theorem, which is a more precise version of the aforementioned conjecture 3 in \cite{CVZ}:

\begin{theorem}\label{thm:main}
There is a constant $c>0$ such that
\begin{equation}\label{eq:main-scales}
 M_n\sim Q_n\sim c\ph^n,\qquad |C_n|\sim c\ph^{n-1}.
\end{equation}
Consequently,
\begin{equation}\label{eq:main-limits}
 \lim_{n\to\infty}\frac{|C_{n+1}|}{|C_n|}=\ph,
 \qquad
 \lim_{n\to\infty}\frac{|C_n|}{M_n}=\frac1\ph.
\end{equation}
\end{theorem}

The proof of the  theorem uses the cited exhaustion result in \cite{CVZ},
an almost-all theorem for Goldbach representations with nearly equal primes,
and two standard estimates for exponential sums over primes. A very particular
interval-restricted representation lemma needed here is proved in
Appendix~\ref{app:binary}. The proof was constructed with algorithmic assistance from OpenAI Astra, 
and its correctness has been formally verified using the Lean proof assistant (code available from github).

\section{Heuristics and Proof overview}

\begin{enumerate}
\item \textbf{Parity gives the Fibonacci upper bound.}
If a sum is composite, division by its least prime factor prevents the
output from exceeding the preceding maximum. A new maximum must therefore
be prime. An odd prime is a sum of an odd and an even input, and the even
input cannot exceed the maximum from one generation earlier. This gives
$M_{n+1}\le M_n+M_{n-1}$.

\item \textbf{Complete prime prefixes reproduce at the Fibonacci scale.}
When every prime up to $X$ is present, almost-all Goldbach representations
produce all but $O_A(X/\log^A X)$ integers up to $X$ one generation later.
These integers, paired with primes near a larger cutoff, extend the complete
prime prefix by an almost Fibonacci step. Summable relative losses then give
positive constants $c,d$ with $M_n\sim c\ph^n$ and $Q_n\sim d\ph^n$.

\item \textbf{A gap between these constants produces extra primes.}
If $d<c$, an earlier maximum supplies a generated prime $t$ beyond the
predicted complete prefix at a suitable scale. The equation
$2p+q=2u-t$, with $p$ and $q$ in specified complete prime intervals,
generates $u$ through the two operations
$(t+q)/2$ and $p+(t+q)/2$. An almost-all representation estimate therefore
produces almost every prime in an interval beyond the predicted prefix.

\item \textbf{Two small exceptional sets eliminate the gap.}
Take the first missing prime two generations later. It has many possible
decompositions $r=u+(r-u)$ with $u$ in the extra interval. Only a small
number of these $u$ are missing, and only a small number of their complementary
integers $r-u$ are missing. At least one pair is present, forcing $r$ to be
generated and contradicting its definition. Hence $d=c$.

\item \textbf{The preceding maximum determines the cardinality.}
Since $Q_{n-1}\sim M_{n-1}$, the Goldbach filling estimate makes
$C_n$ almost complete below $M_{n-1}$. Every element above that threshold
is prime, and those primes contribute only $o(M_{n-1})$ elements.
Thus $|C_n|\sim M_{n-1}$, yielding both the golden-ratio growth and
the limiting density $1/\ph$.
\end{enumerate}

\section{Analytic preliminaries}

All logarithms are natural. We will denote cardinality of a set $X$ either by $|X|$ or by $\# X$. \ 
The notations $f=O_A(g)$ and $f\ll_A g$ mean $|f|\le C_Ag$ for all sufficiently large values of the relevant
parameter. We will write $f\asymp g$ when both $f\ll g$ and $g\ll f$ hold, and $f\sim g$
when $f/g\to1$. 

Let $\Pp$ denote the set of primes and $\pi(X)$
the number of primes at most $X$. Intervals used to count missing
elements are to be intersected with $\N$. Define
\[
 H_j(X)=\#\bigl(\{1,\ldots,\floor X\}\setminus C_j\bigr).
\]

\begin{lemma}[Prime counts]\label{lem:pnt}
As $X\to\infty$,
\begin{equation}\label{eq:short-primes}
 \#\left(\Pp\cap\left[X-\frac{X}{\log^2X},X\right]\right)
 \sim\frac{X}{\log^3X}.
\end{equation}
For fixed $0<\alpha<\beta$,
\begin{equation}\label{eq:proportional-primes}
 \#\bigl(\Pp\cap[\alpha X,\beta X]\bigr)
 \sim\frac{(\beta-\alpha)X}{\log X},
\end{equation}
Also $\pi(X)\sim X/\log X$. The prime immediately preceding a large
real $X$ is $X+O(X/\log^2X)$, and the prime immediately following $X$
is $X+o(X)$.
\end{lemma}

\begin{proof}
The $q=1$, zero-frequency case of \cite[Lemma 5 (Major arc estimates), equation (26)]{LyallRice} or \cite[Proposition~24]{Tao},  
which both follow from the Siegel-Walfisz theorem on primes in arithmetic progressions, yields
for every fixed $D>0$, that
\[
 \sum_{m\le X}\Lambda(m)=X+O_D(X/\log^D X).
\]
where  $\Lambda$ is the  Mangoldt function.
The prime powers with exponent at least two have total weight
$O(\sqrt X\log^2X)$: there are at most $\log_2X$ possible exponents,
and for each exponent at most $\sqrt X$ bases, each of weight at most
$\log X$. Thus
\begin{equation}\label{eq:theta}
 \vartheta(X):=\sum_{p\le X}\log p
 =X+O_D(X/\log^D X).
\end{equation}
Subtract this formula at $X$ and $X-h$, where $h=X/\log^2X$, taking
$D=4$. It gives total prime weight $h+O(X/\log^4X)\sim h$ in
$(X-h,X]$. Every such prime has logarithm $\sim\log X$, uniformly,
which proves \eqref{eq:short-primes}; changing an endpoint adds at most
one prime. The same subtraction at $\beta X$ and $\alpha X$ proves
\eqref{eq:proportional-primes}. Partial summation gives
\[
 \pi(X)=\frac{\vartheta(X)}{\log X}
       +\int_2^X\frac{\vartheta(t)}{t\log^2t}\,dt
 =\frac{X}{\log X}+O(X/\log^2X).
\]
For the integral bound, split at $\sqrt X$ and use $\vartheta(t)\ll t$.
The preceding-prime claim follows from the nonemptiness of the interval
in \eqref{eq:short-primes}. For every fixed $\epsilon>0$,
\eqref{eq:proportional-primes} supplies a prime between $X$ and
$(1+\epsilon)X$ for all sufficiently large $X$. Letting $\epsilon$
be arbitrary proves the following-prime claim.
\end{proof}

\begin{lemma}[Goldbach filling]\label{lem:filling}
Fix $A>0$. If $C_j$ contains every prime at most $X$, then
\begin{equation}\label{eq:filling}
 H_{j+1}(X)\ll_A\frac{X}{\log^A X}.
\end{equation}
The implicit constant is independent of $j$.
\end{lemma}

\begin{proof}
The corollary to Theorem~1 in \cite{CL}, with exponent $5/8+1/8=3/4$, or equations (1.5), (1.6) and Theorem 1.2 in \cite{BakerHarman1998}
yield that all but
$O_A(Y/\log^A Y)$ integers $m\in[Y,2Y]$ have a representation
\begin{equation}\label{eq:near-goldbach}
 2m=p+q,\qquad p,q\in\Pp,\qquad |p-m|,|q-m|\le m^{3/4}.
\end{equation}
Here their parameter for the even number $2m$ is rescaled by a factor
of two, which does not change the order of the exceptional-set bound.

Consider $\sqrt X<m\le X-X^{3/4}$. Cover this range by dyadic
intervals $[X/2^{r+1},X/2^r]$ that meet it. Each lower endpoint $Y$
is at least $\sqrt X/2$, their sum is at most $X$, and therefore
$\log Y\ge \tfrac13\log X$ for large $X$. The total number of
exceptions to \eqref{eq:near-goldbach} is consequently at most
\[
 C_A\sum_r\frac{X/2^{r+1}}{\log^A(X/2^{r+1})}
 \le 3^AC_A\frac{X}{\log^A X}.
\]
For every nonexceptional $m$ in this range,
$p,q\le m+m^{3/4}\le X$. Hence $p,q\in C_j$, and, as $m\ge2$,
the number $2m$ is composite with least prime factor $2$. Thus
$s(p+q)=m\in C_{j+1}$. The omitted bottom and top ranges contain
$O(\sqrt X+X^{3/4}+1)$ integers. This is $O_A(X/\log^A X)$,
proving the lemma.
\end{proof}

\begin{lemma}[Two primes in specified intervals]\label{lem:binary}
Fix $0<\alpha<\beta$, $0<\gamma<\delta$, $\eta>0$, $B>0$, and $A>0$.
Put
\[
 I_X=[\alpha X,\beta X],\qquad J_X=[\gamma X,\delta X],
 \qquad
 \mathcal J_X(N)=\meas\{v\in I_X:N-2v\in J_X\}.
\]
Among the odd integers $1\le N\le BX$ with
$\mathcal J_X(N)\ge\eta X$, at most $O(X/\log^A X)$ fail to have
a representation
\begin{equation}\label{eq:binary-rep}
 N=2p+q,\qquad p\in\Pp\cap I_X,\quad q\in\Pp\cap J_X.
\end{equation}
The implicit constant may depend on all the displayed fixed parameters.
\end{lemma}

The proof of this lemma is given in Appendix~\ref{app:binary}. It uses no assertion
about every individual odd target; the exceptional set is essential.

\section{Fibonacci growth and the two asymptotic constants}

\begin{lemma}[Structure and upper bound]\label{lem:structure}
For $n\ge1$, the maximum $M_n$ is prime and
\begin{equation}\label{eq:structure}
 C_n\subseteq\{1,\ldots,M_{n-1}\}
       \cup\bigl(\Pp\cap[1,M_n]\bigr).
\end{equation}
For $n\ge1$,
\begin{equation}\label{eq:fib-upper}
 M_{n+1}\le M_n+M_{n-1},\qquad M_n\le F_{n+2},
\end{equation}
where $F_0=0$, $F_1=1$, and $F_{r+2}=F_{r+1}+F_r$.
\end{lemma}

\begin{proof}
If $a,b\le M_{n-1}$ and $a+b$ is composite, then
\[
 s(a+b)=\frac{a+b}{\lpf(a+b)}\le\frac{a+b}{2}\le M_{n-1}.
\]
If $a+b$ is prime, the output is prime. Previously retained elements
are also at most $M_{n-1}$, proving \eqref{eq:structure}. Since
$M_1=2$, a maximum that changes can change only to a prime, and a
maximum that does not change remains prime. Moreover $M_2=3$, so
$M_n$ is odd for $n\ge2$.

For $n\ge2$, every even element of $C_n$ is at most $M_{n-1}$:
this follows from \eqref{eq:structure} for composites, and from
$2\le M_{n-1}$ for the only even prime. If $M_{n+1}>M_n$, its
generating sum is an odd prime, so one input is even and at most
$M_{n-1}$, and the other is at most $M_n$. If the maximum does not
increase, the same bound is immediate. The case $n=1$ is
$M_2=3=M_1+M_0$. Induction from $M_0=1$, $M_1=2$ now proves
the Fibonacci bound. In particular $M_n\ll\ph^n$, for example by
the formula $F_r=(\ph^r-(-\ph)^{-r})/\sqrt5$.
\end{proof}

\begin{lemma}[Extension of a complete prime prefix]\label{lem:extension}
Fix $B\ge1$. Suppose $Y\le X\le BY$, $C_n$ contains every prime
at most $X$, and $C_{n-1}$ contains every prime at most $Y$.
For sufficiently large $Y$, depending only on $B$, the set $C_{n+1}$
contains every prime at most
\begin{equation}\label{eq:extension}
 X+Y-\frac{X}{\log^2X}.
\end{equation}
\end{lemma}

\begin{proof}
Set $h=X/\log^2X$. For large $Y$, $h<Y$. Lemma~\ref{lem:filling}
gives $H_n(Y)\ll Y/\log^4Y$. Let $r$ be a prime with
$X<r\le X+Y-h$. For each prime $p\in[X-h,X]$, the integer
$a=r-p$ is positive and satisfies
\[
 a\le(X+Y-h)-(X-h)=Y.
\]
Distinct choices of $p$ give distinct $a$. By Lemma~\ref{lem:pnt},
the number of candidate primes is asymptotic to $X/\log^3X$, which
exceeds $H_n(Y)$ for sufficiently large $Y$: indeed $X\asymp_B Y$
and $\log X\sim\log Y$. Some candidate therefore has $a\in C_n$.
Also $p\in C_n$, and $s(p+a)=r$ because $r$ is prime. The primes
at most $X$ are retained, completing the proof.
\end{proof}

\begin{proposition}\label{prop:rough-scales}
We have $Q_n\asymp M_n\asymp\ph^n$.
\end{proposition}

\begin{proof}
Choose a fixed real $L$ large enough for Lemma~\ref{lem:extension}
with $B=2$, and such that $1/\log^2L\le1/4$.
By \eqref{eq:exhaustion} and nesting there is an $N$ with
$\{1,\ldots,\ceil L\}\subseteq C_N$. Define
\begin{equation}\label{eq:L-recursion}
 L_0=L_1=L,\qquad
 L_{k+1}=L_k+L_{k-1}-\epsilon_kL_k,\qquad
 \epsilon_k=\frac1{\log^2L_k}\quad(k\ge1).
\end{equation}
Inductively $L_{k-1}\le L_k\le2L_{k-1}$ and $L_k\ge L$.
Indeed, under these inequalities, $\epsilon_kL_k\le L_{k-1}/2$,
and therefore
\begin{equation}\label{eq:L-growth}
 \frac54L_k\le L_k+\frac12L_{k-1}
 \le L_{k+1}\le L_k+L_{k-1}\le2L_k.
\end{equation}
The base case is $L_0=L_1$. Lemma~\ref{lem:extension}, applied
successively, shows that every prime up to $L_k$ lies in $C_{N+k}$.
The two initial cases follow from the choice of $N$ and retention.

By \eqref{eq:L-growth}, $L_k\ge L(5/4)^{k-1}$ for $k\ge1$,
so $\sum_{k\ge1}\epsilon_k<\infty$. Set $u_k=L_k/\ph^k$.
Since $\ph^{-1}+\ph^{-2}=1$, \eqref{eq:L-recursion} implies
\[
 u_{k+1}\ge(1-\epsilon_k)
 \left(\frac{u_k}{\ph}+\frac{u_{k-1}}{\ph^2}\right).
\]
For $m_k=\min(u_{k-1},u_k)$ this gives
$m_{k+1}\ge(1-\epsilon_k)m_k$. Thus
\[
 m_k\ge m_1\prod_{j=1}^{k-1}(1-\epsilon_j)
 \ge m_1\exp\left(-2\sum_{j\ge1}\epsilon_j\right)>0,
\]
where $\log(1-z)\ge-2z$ for $0\le z\le1/2$.
Hence $L_k\gg\ph^k$. The largest prime at most $L_k$ is
$\sim L_k$ by Lemma~\ref{lem:pnt}, and is at most $Q_{N+k}$.
Consequently $Q_n\gg\ph^n$. Together with $Q_n\le M_n\ll\ph^n$,
this proves the proposition.
\end{proof}

\begin{lemma}[A contracting recurrence]\label{lem:contraction}
Suppose a bounded real sequence $(x_n)$ satisfies
$x_n+\rho x_{n-1}\to t$ for some $0<\rho<1$.
Then $x_n\to t/(1+\rho)$.
\end{lemma}

\begin{proof}
Put $y_n=x_n-t/(1+\rho)$. Then $y_n=-\rho y_{n-1}+o(1)$.
The finite number $L=\limsup_n|y_n|$ satisfies $L\le\rho L$,
so $L=0$.
\end{proof}

\begin{proposition}\label{prop:constants}
There are constants $0<d\le c$ such that
\begin{equation}\label{eq:constants}
 M_n\sim c\ph^n,\qquad Q_n\sim d\ph^n.
\end{equation}
\end{proposition}

\begin{proof}
Define
\[
 T_n=\ph^{-n}\left(M_n+\frac{M_{n-1}}\ph\right).
\]
Using $\ph-\ph^{-1}=1$ and \eqref{eq:fib-upper}, we obtain
\[
 T_{n+1}-T_n
 =\ph^{-(n+1)}(M_{n+1}-M_n-M_{n-1})\le0.
\]
Proposition~\ref{prop:rough-scales} bounds $T_n$ away from zero,
so $T_n\to t>0$. Apply Lemma~\ref{lem:contraction} to
$x_n=M_n/\ph^n$ and $\rho=\ph^{-2}$. It gives
$M_n/\ph^n\to c=t/(1+\ph^{-2})>0$.

The sequence $Q_n$ is nondecreasing, and
Proposition~\ref{prop:rough-scales} bounds $Q_n/Q_{n-1}$ above by
a fixed constant. Apply Lemma~\ref{lem:extension} with
$X=Q_n$, $Y=Q_{n-1}$. All primes up to
$Z=Q_n+Q_{n-1}-Q_n/\log^2Q_n$ belong to $C_{n+1}$.
Since $Z\asymp Q_n$, Lemma~\ref{lem:pnt} shows that the largest
prime at most $Z$ is at least $Z-O(Q_n/\log^2Q_n)$. Therefore
\begin{equation}\label{eq:Q-lower-recursion}
 Q_{n+1}\ge Q_n+Q_{n-1}-O(Q_n/\log^2Q_n).
\end{equation}
Set
\[
 V_n=\ph^{-n}\left(Q_n+\frac{Q_{n-1}}\ph\right).
\]
It is bounded and bounded away from zero. Since
$Q_n\asymp\ph^n$ and $\log Q_n=n\log\ph+O(1)$,
\eqref{eq:Q-lower-recursion} gives, for fixed $D>0$ and all $n\ge n_0$,
\[
 V_{n+1}-V_n\ge-\frac{D}{n^2}.
\]
The sequence $V_n+D\sum_{j=n_0}^{n-1}j^{-2}$ is bounded and
nondecreasing, hence convergent. Since the added sums converge,
$V_n\to v>0$. Lemma~\ref{lem:contraction} now gives
$Q_n/\ph^n\to d=v/(1+\ph^{-2})>0$. Finally $Q_n\le M_n$
implies $d\le c$.
\end{proof}

\section{Approximate completeness of the prime prefix}

\begin{proposition}\label{prop:equal}
The constants in Proposition~\ref{prop:constants} are equal:
$d=c$. In particular $Q_n/M_n\to1$.
\end{proposition}

\begin{proof}
Assume $d<c$. Choose a fixed integer $k\ge0$ such that
\begin{equation}\label{eq:choose-T}
 d<T:=c\ph^{-k}\le\ph d.
\end{equation}
For example, $k=\ceil{\log(c/d)/\log\ph}-1$ has these properties.
Let
\[
 X=\ph^{n-1},\qquad t_n=M_{n-1-k},\qquad
 \Delta=T-d,\qquad \epsilon=\Delta/16.
\]
All these constants except $X$ and $t_n$ are now fixed.
For sufficiently large $n$, the integer $t_n$ is an odd prime in
$C_{n-1}$, by Lemma~\ref{lem:structure} and retention. Moreover,
\begin{equation}\label{eq:three-scales}
 t_n=TX+o(X),\qquad Q_{n-1}=dX+o(X),\qquad
 Q_n=\ph dX+o(X).
\end{equation}

Define now
\begin{equation}\label{eq:three-intervals}
 \begin{aligned}
 I_X&=[(\ph d-2\epsilon)X,(\ph d-\epsilon)X],\\
 J_X&=[(2d-T+3\epsilon)X,(2d-T+9\epsilon)X],\\
 U_X&=[(\ph^2d+\epsilon)X,(\ph^2d+2\epsilon)X].
 \end{aligned}
\end{equation}
All endpoints are positive. In particular, $T\le\ph d<2d$ makes
the lower endpoint of $J_X$ positive, and
$\epsilon\le(\ph-1)d/16$ makes that of $I_X$ positive.
The upper endpoint of $I_X$ is below $Q_n$ for large $n$, and
the upper endpoint of $J_X$ is
\[
 (2d-T+9\epsilon)X=(d-7\epsilon)X<Q_{n-1}.
\]
It follows that every prime in $I_X$ belongs to $C_n$, and every
prime in $J_X$ belongs to $C_{n-1}$.

For an integer $u\in U_X$, consider the target $N=2u-t_n$.
It is odd, and it lies in $[1,BX]$ for a fixed sufficiently large
$B$: its leading lower coefficient is
$2\ph^2d+2\epsilon-T>0$, and its upper coefficient is fixed.
For every real $p\in I_X$, the value $q=N-2p$ satisfies
\begin{equation}\label{eq:q-margins}
 (2d-T+4\epsilon)X+o(X)
 \le q\le
 (2d-T+8\epsilon)X+o(X).
\end{equation}
To verify these bounds, use the lower endpoint of $U_X$ and the
upper endpoint of $I_X$ for the lower bound, and the opposite
endpoints for the upper bound, together with $\ph^2-\ph=1$.
The error is exactly $-(t_n-TX)$, independent of $u$ and $p$.
For large $n$ its absolute value is less than $\epsilon X/2$.
Thus every such $q$ lies in $J_X$, and
\[
 \meas\{p\in I_X:2u-t_n-2p\in J_X\}=|I_X|=\epsilon X.
\]

Lemma~\ref{lem:binary} applies with these fixed intervals and
$\eta=\epsilon$. Since $u\mapsto2u-t_n$ is injective, all but
$O_A(X/\log^A X)$ integers $u\in U_X$ have a representation
\begin{equation}\label{eq:bridge}
 2p+q=2u-t_n,\qquad p\in\Pp\cap I_X,\quad q\in\Pp\cap J_X.
\end{equation}
If such a $u$ is itself prime, both operations in
\begin{equation}\label{eq:two-operations}
 b=\frac{t_n+q}{2}=s(t_n+q)\in C_n,
 \qquad u=p+b=s(p+b)\in C_{n+1}
\end{equation}
are valid. Indeed $t_n$ and $q$ are odd primes in $C_{n-1}$,
so their sum is composite with least prime factor $2$; and
$p\in C_n$, while the final sum $u$ is prime. Consequently
\begin{equation}\label{eq:extra-primes}
 \#\bigl((\Pp\cap U_X)\setminus C_{n+1}\bigr)
 \ll_A X/\log^A X.
\end{equation}

%\medskip\noindent\textit{Contradicting the first missing prime.}
Let $r_n=R_{n+2}$, the first prime absent from $C_{n+2}$.
It is the prime immediately after $Q_{n+2}$. By
Lemma~\ref{lem:pnt} and Proposition~\ref{prop:constants},
\begin{equation}\label{eq:first-missing-scale}
 r_n=\ph^3dX+o(X).
\end{equation}
For a prime $u\in U_X$, put $a=r_n-u$. The identity
$\ph^3-\ph^2=\ph$ gives
\[
 (\ph d-2\epsilon)X+o(X)
 \le a\le(\ph d-\epsilon)X+o(X).
\]
The errors here are uniform in $u$. By \eqref{eq:three-scales}
and the fixed positive margins, $1\le a\le Q_n$ for large $n$.

Lemma~\ref{lem:pnt} supplies
\[
 \#(\Pp\cap U_X)\sim\frac{\epsilon X}{\log X}
\]
candidate primes. At most $O(X/\log^4X)$ of them are absent from
$C_{n+1}$, by \eqref{eq:extra-primes} with $A=4$. Among the
remaining candidates, at most $H_{n+1}(Q_n)$ have
$r_n-u\notin C_{n+1}$, since distinct $u$ give distinct differences.
Lemma~\ref{lem:filling} and $Q_n\asymp X$ give
\[
 H_{n+1}(Q_n)\ll\frac{Q_n}{\log^4Q_n}
 \ll\frac{X}{\log^4X}.
\]
The combined exclusions are $o(X/\log X)$, so at least one candidate
has both $u\in C_{n+1}$ and $r_n-u\in C_{n+1}$. For that candidate,
\[
 s\bigl(u+(r_n-u)\bigr)=r_n\in C_{n+2},
\]
because $r_n$ is prime. This contradicts its definition and proves
$d=c$.
\end{proof}

\section{Cardinalities, limiting density and proof of Theorem~\ref{thm:main}}

\begin{proof}[Completion of the proof of Theorem~\ref{thm:main}]
Propositions~\ref{prop:constants} and \ref{prop:equal} give
$M_n\sim Q_n\sim c\ph^n$. Since both $M_{n-1}$ and $Q_{n-1}$
are integers, Lemma~\ref{lem:filling} gives
\begin{align*}
 H_n(M_{n-1})
 &\le M_{n-1}-Q_{n-1}+H_n(Q_{n-1})\\
 &\le M_{n-1}-Q_{n-1}
       +O\left(\frac{Q_{n-1}}{\log^4Q_{n-1}}\right)
 =o(M_{n-1}).
\end{align*}
On the other hand, \eqref{eq:structure} implies
\[
 M_{n-1}-H_n(M_{n-1})\le |C_n|\le M_{n-1}+\pi(M_n).
\]
Because $M_n/M_{n-1}\to\ph$ and
$\pi(M_n)=O(M_n/\log M_n)=o(M_{n-1})$, these two bounds give
$|C_n|\sim M_{n-1}\sim c\ph^{n-1}$. Taking successive ratios and
dividing by $M_n$ proves \eqref{eq:main-limits}.
\end{proof}

\begin{remark}
The assertion $Q_n/M_n\to1$ says that the \emph{complete} prime
prefix reaches a relative distance $o(1)$ from the maximum. For the
integers we prove $H_n(M_{n-1})=o(M_{n-1})$, allowing a sparse set
of holes. In particular, the proof does not require the solid core
$K_n=\max\{K:\{1,\ldots,K\}\subseteq C_n\}$ to satisfy
$K_n\sim M_{n-1}$.
\end{remark}

\appendix
\section{Proof of the interval-restricted binary lemma}\label{app:binary}

We prove Lemma~\ref{lem:binary}, retaining its fixed parameters.
Only the two analytic estimates stated next are used without proof.
All subsequent interval, local-factor, and exceptional-set calculations
are included. Write $e(z)=\exp(2\pi iz)$, and use $\tot$ for Euler's
totient, to distinguish it from the golden ratio $\ph$.

\subsection{The analytic estimates and representation count}

The major-arc estimate we use is \cite[Lemma 5 (Major arc estimates)]{LyallRice}  or \cite[Proposition~24]{Tao}:
for an interval $\mathcal I\subseteq[1,x]$, a reduced fraction $a/q$,
and real $\xi$, one has, for every $D>0$,
\begin{equation}\label{eq:tao-input}
 \sum_{m\in\mathcal I\cap\Z}\Lambda(m)
       \ee{(a/q+\xi/q)m}
 =\frac{\mu(q)}{\tot(q)}\int_{\mathcal I}\ee{\xi v/q}\,dv
 +O_D\left((q+x|\xi|)\frac{x}{\log^D x}\right).
\end{equation}
Its implicit constant need not be effective. Here $\mu$ is the
M\"obius function and $\Lambda$ is the von Mangoldt function.

The minor-arc estimate is the following consequence of
\cite[Lemma~1]{Kumchev}: if $(a,q)=1$, $1\le q\le y$, and
$|\theta-a/q|\le q^{-2}$, then
\begin{equation}\label{eq:vaughan-input}
 \left|\sum_{p\le y}(\log p)e(p\theta)\right|
 \ll\bigl(yq^{-1/2}+y^{4/5}+(yq)^{1/2}\bigr)(\log(2y))^4.
\end{equation}
The additional factor $1+q^2|\theta-a/q|$ in the cited lemma is
at most $2$ under this hypothesis.

Define
\[
 S_I(\theta)=\sum_{p\in\Pp\cap I_X}(\log p)e(p\theta),\qquad
 S_J(\theta)=\sum_{q\in\Pp\cap J_X}(\log q)e(q\theta).
\]
Orthogonality of the functions $e(m\theta)$ for integers $m$ gives
\begin{equation}\label{eq:weighted-count}
 R_X(N):=\int_0^1 S_I(2\theta)S_J(\theta)e(-N\theta)\,d\theta
 =\sum_{\substack{p\in\Pp\cap I_X, q\in\Pp\cap J_X\\2p+q=N}}
       (\log p)(\log q).
\end{equation}
Thus $R_X(N)>0$ is equivalent to a representation of the required kind.
Fix
\begin{equation}\label{eq:choose-P}
 K=A+20,\qquad P=(\log X)^K.
\end{equation}
For reduced fractions $a/q$ with $1\le q\le P$ and $0\le a<q$,
let
\[
 \mathfrak M(a,q)=
 \{\theta\in\R/\Z:\|\theta-a/q\|\le P/X\},
 \qquad
 \mathfrak M=\bigcup_{q\le P}\ \bigcup_{\substack{0\le a<q\\(a,q)=1}}
                 \mathfrak M(a,q).
\]
Here $\|\cdot\|$ is circular distance to $0$; the fraction $0/1$
is included. Distinct fractions are at circular distance at least
$P^{-2}$, so these arcs are disjoint once $2P^3<X$.
Their total measure is $O(P^3/X)$. Set
$\mathfrak m=(\R/\Z)\setminus\mathfrak M$, and denote the two
parts of \eqref{eq:weighted-count} by $R_{\mathfrak M}(N)$ and
$R_{\mathfrak m}(N)$.

\subsection{The minor arcs}

Put $D_0=\ceil{X/P}$. Dirichlet approximation gives, for every
$\theta\in\mathfrak m$, a reduced fraction $a/q$ with
\[
 1\le q\le D_0,\qquad
 |\theta-a/q|\le\frac1{qD_0}\le\frac{P}{qX}.
\]
Distances may be interpreted modulo $1$ by changing $a$ by a
multiple of $q$. If $q\le P$, this places $\theta$ in
$\mathfrak M(a,q)$, a contradiction. Thus, for large $X$,
\begin{equation}\label{eq:minor-denominator}
 P<q\le\frac{2X}{P},\qquad |\theta-a/q|\le q^{-2}.
\end{equation}
Apply \eqref{eq:vaughan-input} to the two initial prime sums defining
$S_J$. Their endpoints are fixed positive multiples of $X$, so
$q$ is below both endpoints for all sufficiently large $X$.
An endpoint convention changes the sum by at most $O(\log X)$.
Using \eqref{eq:minor-denominator}, we obtain
\begin{equation}\label{eq:minor-sup}
 \sup_{\theta\in\mathfrak m}|S_J(\theta)|
 \ll XP^{-1/2}(\log X)^4+X^{4/5}(\log X)^4.
\end{equation}
Parseval's identity and \eqref{eq:theta} give
\[
 \int_0^1|S_I(2\theta)|^2\,d\theta
 =\sum_{p\in\Pp\cap I_X}(\log p)^2\ll X\log X.
\]
The factor $2$ does not affect orthogonality, since $2(p-p')$ is
a nonzero integer when $p\ne p'$. Bessel's inequality, applied to
$\ind_{\mathfrak m}(\theta)S_I(2\theta)S_J(\theta)$, therefore yields
\begin{align}
 \sum_{N\in\Z}|R_{\mathfrak m}(N)|^2
 &\le\int_{\mathfrak m}|S_I(2\theta)S_J(\theta)|^2\,d\theta\notag\\
 &\ll X^3P^{-1}(\log X)^9+X^{13/5}(\log X)^9
 \ll_A \frac{X^3}{\log^{A+2}X}.\label{eq:minor-L2}
\end{align}
For the last inequality, the first term has logarithmic exponent
$9-K=-A-11$, and the second satisfies
$X^{-2/5}\log^{A+11}X\to0$ after division by $X^3/\log^{A+2}X$.
In particular, for each fixed $\lambda>0$,
\begin{equation}\label{eq:minor-exceptions}
 \#\{N\in\Z:|R_{\mathfrak m}(N)|>\lambda X\}
 \ll_{A,\lambda}\frac{X}{\log^{A+2}X}.
\end{equation}

\subsection{The major arcs and the real solution measure}

For real $\zeta$, put
\[
 V_I(\zeta)=\int_{I_X}e(v\zeta)\,dv,\qquad
 V_J(\zeta)=\int_{J_X}e(v\zeta)\,dv.
\]
On an arc $\theta=a/q+\zeta$, $|\zeta|\le P/X$, let
$q'=q/(q,2)$, the denominator after reducing $2a/q$.
Formula \eqref{eq:tao-input} implies, for every fixed $E>0$,
\begin{equation}\label{eq:major-approx}
 \begin{aligned}
 S_J(a/q+\zeta)
  &=\frac{\mu(q)}{\tot(q)}V_J(\zeta)+O_{E,K}(X/\log^E X),\\
 S_I(2a/q+2\zeta)
  &=\frac{\mu(q')}{\tot(q')}V_I(2\zeta)+O_{E,K}(X/\log^E X).
 \end{aligned}
\end{equation}
Here is the required uniformity check. Choose $x=CX$ with fixed
$C>\max(\beta,\delta,1)$. In the first line of \eqref{eq:major-approx},
use $\xi=q\zeta$ in \eqref{eq:tao-input}; in the second, use
$\xi=2q'\zeta$ with denominator $q'$. In both cases
$q+x|\xi|=O(P^2)$, with $q'$ replacing $q$ when appropriate.
Taking exponent $D=E+2K+2$ in \eqref{eq:tao-input} gives the
stated error. Removing prime powers costs at most
$O(\sqrt X\log^2X)=O_E(X/\log^E X)$, as in Lemma~\ref{lem:pnt}.

Take $E=A+3K+10$. The sums $S_I,S_J$ and the integrals $V_I,V_J$
are $O(X)$; for the sums, this follows from \eqref{eq:theta}.
Since the major arcs have measure $O(P^3/X)$, multiplying
\eqref{eq:major-approx} and integrating causes total error
\begin{equation}\label{eq:major-product-error}
 O\bigl(XP^3/\log^E X\bigr)=O(X/\log^{A+10}X).
\end{equation}
Define the Ramanujan sum and its coefficient by
\[
 c_q(N)=\sum_{\substack{0\le a<q\\(a,q)=1}}e(-Na/q),\qquad
 w_q=\frac{\mu(q)\mu(q')}{\tot(q)\tot(q')}.
\]
The main term is
\[
 \sum_{q\le P}w_qc_q(N)
 \int_{-P/X}^{P/X}V_I(2\zeta)V_J(\zeta)e(-N\zeta)\,d\zeta.
\]
For any interval $[a,b]$ and nonzero $\zeta$,
$|\int_a^b e(v\zeta)\,dv|\le\min(b-a,1/(\pi|\zeta|))$.
Consequently
\[
 |V_I(2\zeta)V_J(\zeta)|\ll\min(X^2,|\zeta|^{-2}),
 \qquad
 \int_{|\zeta|>P/X}|V_I(2\zeta)V_J(\zeta)|\,d\zeta\ll X/P.
\]
Using $|c_q(N)|\le\tot(q)$, the error from extending every integral
to $\R$ is at most
\begin{equation}\label{eq:integral-tail}
 \frac{CX}{P}\sum_{q\le P}\frac1{\tot(q')}
 \ll\frac{X\log^2(2P)}{P}.
\end{equation}
Indeed each integer $q'$ arises from at most two values of $q$,
and $r/\tot(r)\le\tau(r)$, where $\tau$ counts divisors. Hence
\[
 \sum_{q\le P}\frac1{\tot(q')}
 \le2\sum_{r\le P}\frac{\tau(r)}r
 =2\sum_{ab\le P}\frac1{ab}
 \le2(1+\log P)^2.
\]
The inequality $r/\tot(r)\le\tau(r)$ follows prime by prime:
for $\ell^a\parallel r$, $\ell/(\ell-1)\le2\le a+1$.

For clarity, Fourier inversion introduces no unspecified geometric
constant here. The continuous function
\[
 f(y)=\left(\tfrac12\ind_{2I_X}*\ind_{J_X}\right)(y)
 =\int_{I_X}\ind_{J_X}(y-2v)\,dv
\]
has Fourier transform $V_I(2\zeta)V_J(\zeta)$ with the positive
exponential convention used above. That product is integrable by
the displayed bound, and $f(N)=\mathcal J_X(N)$. Fourier inversion gives
\begin{equation}\label{eq:singular-integral}
 \int_\R V_I(2\zeta)V_J(\zeta)e(-N\zeta)\,d\zeta
 =\mathcal J_X(N).
\end{equation}
Write
\[
 \mathfrak S_P(N)=\sum_{q\le P}w_qc_q(N).
\]
Combining \eqref{eq:major-product-error}--\eqref{eq:singular-integral},
and using $K=A+20$, proves the uniform formula
\begin{equation}\label{eq:major-final}
 R_{\mathfrak M}(N)
 =\mathfrak S_P(N)\mathcal J_X(N)+O_A(X/\log^{A+2}X).
\end{equation}
For the second error, note explicitly that
$\log^2(2P)/P=O_A(\log^{-A-2}X)$.

\subsection{An elementary bound for the series tail}

We first prove that, for $Y\ge2$ and $N\ge1$,
\begin{equation}\label{eq:series-tail}
 \sum_{q>Y}\frac{|c_q(N)|}{\tot(q)^2}
 \ll\frac{\log^3(2Y)}{Y}
       F(N),\qquad
 F(N)=\sum_{h\mid N}\left(\frac{h}{\tot(h)}\right)^2.
\end{equation}
The identity
\begin{equation}\label{eq:ramanujan-identity}
 c_q(N)=\sum_{h\mid(q,N)}h\mu(q/h)
\end{equation}
follows by inserting
$\ind_{(a,q)=1}=\sum_{d\mid(a,q)}\mu(d)$ into the defining exponential
sum. The sum over multiples of $d$ is $q/d$ if $q/d\mid N$ and
zero otherwise; set $h=q/d$. In particular,
$|c_q(N)|\le\sum_{h\mid(q,N)}h$.

The product formula for the totient gives
$\tot(hk)\ge\tot(h)\tot(k)$. Therefore, writing $q=hk$ and
interchanging nonnegative sums,
\begin{equation}\label{eq:tail-divisors}
 \sum_{q>Y}\frac{|c_q(N)|}{\tot(q)^2}
 \le\sum_{h\mid N}\frac{h}{\tot(h)^2}
             \sum_{k>Y/h}\frac1{\tot(k)^2}.
\end{equation}
For $z\ge1$ we claim
\begin{equation}\label{eq:totient-tail}
 \sum_{k>z}\frac1{\tot(k)^2}\ll\frac{\log^3(2z)}z.
\end{equation}
Let $\tau_4(k)$ count ordered factorizations $k=abcd$ in positive
integers. Prime by prime,
\[
 \left(\frac{k}{\tot(k)}\right)^2\le\tau(k)^2\le\tau_4(k),
\]
since $(a+1)^2\le\binom{a+3}{3}$ for every integer $a\ge0$.
Furthermore, for $w\ge1$,
\[
 \sum_{k\le w}\tau_4(k)
 =\sum_{abc\le w}\floor{\frac{w}{abc}}
 \le w\left(\sum_{a\le w}\frac1a\right)^3
 \le w(1+\log w)^3.
\]
Split $k>z$ into ranges $2^jz<k\le2^{j+1}z$. These inequalities give
\[
 \sum_{k>z}\frac1{\tot(k)^2}
 \le\sum_{j\ge0}\frac1{(2^jz)^2}
                    \sum_{k\le2^{j+1}z}\tau_4(k)
 \ll\frac1z\sum_{j\ge0}2^{-j}(\log(2z)+j)^3
 \ll\frac{\log^3(2z)}z,
\]
because $\sum_{j\ge0}2^{-j}(j+1)^3$ converges. This proves
\eqref{eq:totient-tail}. For $0\le z<1$, the same sum is $O(1)$,
by the bound at $z=1$ and the single term $k=1$.

For $h\le Y$, apply \eqref{eq:totient-tail} to $z=Y/h$ in
\eqref{eq:tail-divisors}, bounding $\log(2Y/h)$ by $\log(2Y)$.
The resulting contribution for $h$ is at most
\[
 C\frac{\log^3(2Y)}Y\frac{h^2}{\tot(h)^2}.
\]
For $h>Y$, the $O(1)$ bound gives at most $Ch/\tot(h)^2$,
which is also bounded by the same expression after increasing $C$.
This proves \eqref{eq:series-tail}, including absolute convergence
for each fixed $N$.

We will also need its first moment. For large $X$,
\begin{align}
 \sum_{1\le N\le BX}F(N)
 &\le BX\sum_{h\le BX}\frac1h
                  \left(\frac{h}{\tot(h)}\right)^2\notag\\
 &\le BX\sum_{h\le BX}\frac{\tau_4(h)}h
 =BX\sum_{abcd\le BX}\frac1{abcd}
 \le BX(1+\log(BX))^4
 \ll_B X\log^4X.\label{eq:F-mean}
\end{align}

\subsection{The local factors and positivity}

For odd $N$, absolute convergence from \eqref{eq:series-tail}
allows us to compute the full series
$\mathfrak S(N)=\sum_{q\ge1}w_qc_q(N)$. A nonzero coefficient
requires $q$ squarefree. For odd squarefree $q$, we have $q'=q$
and $w_q=\mu(q)^2/\tot(q)^2$. For even squarefree $q=2r$,
with $r$ odd, we have $q'=r$ and
\[
 w_{2r}=-\frac{\mu(r)^2}{\tot(r)^2},\qquad
 c_{2r}(N)=c_2(N)c_r(N)=-c_r(N).
\]
Here $c_q(N)$ is multiplicative in $q$ by
\eqref{eq:ramanujan-identity}, and $c_2(N)=-1$ for odd $N$.
Thus the even term equals the corresponding odd term, and
\begin{equation}\label{eq:full-series}
 \begin{aligned}
 \mathfrak S(N)
 &=2\sum_{\substack{r\ge1\\r\text{ odd}}}
               \frac{\mu(r)^2c_r(N)}{\tot(r)^2}\\
 &=2\prod_{\ell>2}
             \left(1-\frac1{(\ell-1)^2}\right)
       \prod_{\substack{\ell\mid N\\\ell>2}}
             \frac{\ell-1}{\ell-2}.
 \end{aligned}
\end{equation}
All products indexed by $\ell$ here are over primes. To verify
each factor, \eqref{eq:ramanujan-identity} gives
$c_\ell(N)=\ell-1$ if $\ell\mid N$ and $c_\ell(N)=-1$ otherwise.
The factor at a divisor prime is $1+1/(\ell-1)$; dividing this by
$1-1/(\ell-1)^2$ gives $(\ell-1)/(\ell-2)$, as displayed.

In particular,
\begin{equation}\label{eq:s0}
 \mathfrak S(N)\ge s_0:=
 2\prod_{\ell>2}\left(1-\frac1{(\ell-1)^2}\right)>0
 \qquad(N\text{ odd}).
\end{equation}
For positivity of this constant, the factors lie in $(0,1)$,
$1/(\ell-1)^2\le1/4$, and
$\sum_{\ell>2}1/(\ell-1)^2<\infty$. The inequality
$\log(1-z)\ge-2z$ for $0\le z\le1/2$ bounds the product away
from zero.

The odd terms of $\mathfrak S_P(N)$ have $r\le P$, whereas its
even terms have $r\le P/2$. Consequently, for odd $N$,
\[
 |\mathfrak S(N)-\mathfrak S_P(N)|
 \le2\sum_{r>P/2}\frac{|c_r(N)|}{\tot(r)^2}
 \ll\frac{\log^3(2P)}P F(N).
\]
Summing over $1\le N\le BX$ and using \eqref{eq:F-mean}, we see
that the number of odd $N$ in this range for which
$|\mathfrak S(N)-\mathfrak S_P(N)|>s_0/2$ is at most
\begin{equation}\label{eq:series-exceptions}
 O_B\left(\frac{X\log^4X\log^3(2P)}P\right)
 =O_A\left(\frac{X}{\log^A X}\right).
\end{equation}
The last estimate follows from $K=A+20$, since
$\log^3(2P)=O_A((\log\log X)^3)$.

\subsection{Completion of the representation lemma}

Outside the exceptional set in \eqref{eq:series-exceptions}, every
odd target has $\mathfrak S_P(N)\ge s_0/2$. If also
$\mathcal J_X(N)\ge\eta X$, then \eqref{eq:major-final} gives,
for sufficiently large $X$,
\[
 \operatorname{Re}R_{\mathfrak M}(N)
 \ge\frac{s_0\eta X}{2}-\frac{s_0\eta X}{8}
 =\frac{3s_0\eta X}{8}.
\]
By \eqref{eq:minor-exceptions} with $\lambda=s_0\eta/8$,
all but $O_A(X/\log^{A+2}X)$ further targets satisfy
$|R_{\mathfrak m}(N)|\le s_0\eta X/8$. For a target outside
both exceptional sets, \eqref{eq:weighted-count} is a nonnegative
real number and obeys
\[
 R_X(N)=\operatorname{Re}
             \bigl(R_{\mathfrak M}(N)+R_{\mathfrak m}(N)\bigr)
 \ge\frac{s_0\eta X}{4}>0.
\]
The union of the exceptional sets has size $O_A(X/\log^A X)$.
This proves Lemma~\ref{lem:binary}. All interval data were fixed
before $X$ tended to infinity; no uniformity as an interval length
or margin tends to zero has been used. \hfill$\square$

\bigskip
\noindent
\begin{scriptsize}
\textsc{Department of Computer Science, Columbia University,
New York, NY 10027, USA}

\noindent
\textit{Email address:}  rep2159@columbia.edu
\end{scriptsize}

\begin{thebibliography}{9}

\bibitem{BakerHarman1998}
R.~C. Baker and G. Harman,
\emph{The three primes theorem with almost equal summands},
Philos. Trans. Roy. Soc. London Ser. A
\textbf{356} (1998), no.~1738, 763--780.

\bibitem{CVZ}
M. Caragiu, P. A. Vicol, and M. Zaki,
\href{https://www.fq.math.ca/Papers1/55-4/CaragiuVicolZaki03162017.pdf}
{On Conway's subprime function, a covering of $\N$ and an unexpected
appearance of the golden ratio},
\emph{The Fibonacci Quarterly} \textbf{55} (2017), no.~4, 327--331.
%See Theorem~1 and Conjecture~3.

\bibitem{CL}
G. Coppola and M. B. S. Laporta,
%\href{https://www.giovannicoppola.name/files/articoli/8_giovanni_coppola_number_theory_seminar_politecnico_torino_245.pdf}
{On the representation of even integers as sum of two almost equal primes},
\emph{Rendiconti del Seminario Matematico, Universit\`a e Politecnico
di Torino} \textbf{53} (1995), no.~3, 245--252.
%See the corollary to Theorem~1, p.~246.

\bibitem{Github}
Lean4 formalization for the proof of the Golden-ratio growth of Conway's subprime closure,
\url{https://github.com/Renriviera/conway-subprime-golden/}.

\bibitem{GuyKhovanovaSalazar2014}
R.~K. Guy, T. Khovanova, and J. Salazar,
\emph{Conway's Subprime Fibonacci Sequences},
Mathematics Magazine \textbf{87} (2014), no.~5, 323--337.

\bibitem{Kumchev}
A. V. Kumchev,
\href{https://tigerweb.towson.edu/akumchev/a23.pdf}
{On sums of primes from Beatty sequences},
\emph{Integers} \textbf{8} (2008), Article A08, 12 pp.
%See Lemma~1.


\bibitem{LyallRice}
N. Lyall and A. Rice, 
{Polynomial Differences in the Primes}, in \emph{Combinatorial and Additive Number Theory: CANT 2011 and 2012}, 
Springer Proceedings in Mathematics \& Statistics \textbf{101} (2014), 129Ð146.


\bibitem{OEIS}
OEIS Foundation Inc. (2011), The On-Line Encyclopedia of Integer Sequences, A117818, \href{https://oeis.org/A117818}.

\bibitem{openai2026astra}
OpenAI, \emph{GPT-6 Astra},
2026. \url{https://openai.com/index/gpt-6-astra/}.

\bibitem{Roberts}
S. Roberts, 
{Genius At Play: The Curious Mind of John Horton Conway, Audiobook, Jennifer Van Dyck
(Reader)}, 
Bloomsbury Press, 2015.

\bibitem{Tao}
T. Tao,
\href{https://terrytao.wordpress.com/2015/03/30/254a-notes-8-the-hardy-littlewood-circle-method-and-vinogradovs-theorem/}
{254A, Notes 8: The Hardy--Littlewood circle method and Vinogradov's theorem},
lecture notes, 30 March 2015.
%See Proposition~24. Accessed 8 September 2026.

\end{thebibliography}
\end{document}